\documentclass[12pt, reqno]{amsart}

\usepackage{amsfonts,amssymb,latexsym,amsmath, amsxtra}
\usepackage{verbatim}

\usepackage[OT2,T1]{fontenc}
\DeclareSymbolFont{cyrletters}{OT2}{wncyr}{m}{n}
\DeclareMathSymbol{\Sha}{\mathalpha}{cyrletters}{"58}

\theoremstyle{plain}
\newtheorem{theorem}{Theorem}[section]
\newtheorem*{theorem*}{Theorem}
\newtheorem{corollary}[theorem]{Corollary}

\newtheorem{lemma}[theorem]{Lemma}
\newtheorem{proposition}[theorem]{Proposition}
\newtheorem*{conjecture*}{Conjecture}

\theoremstyle{definition}

\theoremstyle{remark}
\newtheorem{remark}{Remark}
\newtheorem{remarks}{Remarks}
\newtheorem*{remark*}{Remark}
\newtheorem*{remarks*}{Remarks}

\numberwithin{equation}{section}

\makeatletter
\newcommand*{\rom}[1]{\expandafter\@slowromancap\romannumeral #1@}
\makeatother

\def\({\left(}
\def\){\right)}

\newcommand{\la}{\lambda}
\allowdisplaybreaks

\begin{document}

\title{An overpartition analogue of the $q$-binomial coefficients}

\author{Jehanne Dousse}
\address{LIAFA, Universite Paris Diderot - Paris 7, 75205 Paris Cedex 13, FRANCE}
\email{jehanne.dousse@liafa.univ-paris-diderot.fr}

 \author{Byungchan Kim}
\address{School of Liberal Arts \\ Seoul National University of Science and Technology \\ 232 Gongreung-ro, Nowon-gu, Seoul,139-743, Korea}
\email{bkim4@seoultech.ac.kr}


\date{\today}
\thanks{This research was supported by the International Research \& Development Program of the National Research Foundation of Korea (NRF) funded by the Ministry of Education, Science and Technology(MEST) of Korea (NRF-2014K1A3A1A21000358), and the STAR program number 32142ZM}

\begin{abstract}
We define an overpartition analogue of Gaussian polynomials (also known as $q$-binomial coefficients) as a generating function for the number of overpartitions fitting inside the $M \times N$ rectangle. We call these new polynomials over Gaussian polynomials or over $q$-binomial coefficients. We investigate basic properties and applications of over $q$-binomial coefficients. In particular, via the recurrences and combinatorial interpretations of over $q$-binomial coefficients, we prove a Rogers-Ramaujan type partition theorem.
\end{abstract}

\maketitle

\section{introduction}
 Gaussian polynomial is defined by 
\[
 G(M,N) ={M+N \brack N}_q = \frac{ (q)_{M+N} }{(q)_{M} (q)_{N}},
\]
where $(a)_{n} = (a;q)_{n} := \prod_{k=1}^{n} (1-aq^{k-1} )$ for $n \in \mathbb{N}_{0} \cup \{ \infty \}$. These polynomials (also known as $q$-binomial coefficients) have played many roles in combinatorics and number theory.  For example, Gaussian polynomials serve as generating functions for the number of inversions in permutations of a multi-set, the number of restricted partitions and the number of $N$ dimensional subspaces of $M+N$ dimensional vector spaces over $\mathbb{F}_q$.

 Our interest in this paper is to study an overpartition analogue of Gaussian polynomials. Recall $G(M,N)$ is the generating function for the number of partitions of $n$ fitting inside an $M \times N$ rectangle, i.e. with largest part $\le M$ and number of parts $\le N$. (For example, see \cite{Abook}.) In this light, we define our overpartition analogue of Gaussian polynomials, which we will call \emph{over $q$-binomial coefficients}, as the generating function for the number of overpartitions fitting inside an $M \times N$ rectangle. An overpartition \cite{LC} is a partition in which the last occurrence of each distinct number may be overlined. For example,  the $8$ overpartitions of $3$ are
\begin{equation*}
\begin{gathered}
3, \overline{3}, 2 + 1, \overline{2} +1, 2 + \overline{1}, \overline{2} + \overline{1}, 1+1+1, 1+1+\overline{1}.
\end{gathered}
\end{equation*}
 Overpartitions have recently played an important role in the construction of weight 3/2 mock modular forms \cite{BL} and in the combinatorial proof of the $_1 \psi_1$ summation formula \cite{Yee}, and their arithmetic properties have been of great interest \cite{acko, BL2, Ma, Tre}. 

 Our first result is an expression for over $q$-binomial coefficients $\overline{{M+N \brack N}}_q$. 
 
\begin{theorem} \label{overGthm}
For positive integers $M$ and $N$, 
\[
\overline{{M+N \brack N}}_q = \sum_{k=0}^{\min\{M,N\}} q^{\frac{k(k+1)}{2}} \frac{(q)_{M+N-k}}{(q)_k(q)_{M-k}(q)_{N-k}}.
\]
\end{theorem}

\begin{remarks}
\begin{enumerate}
\item The above expression can be rewritten by employing $q$-trinomial coefficients 
\[
{a+b+c \brack a,b,c}_q = \frac{ (q)_{a+b+c}}{(q)_a (q)_b (q)_c}.
\]
\item We have an obvious symmetry 
\[
\overline{{M+N \brack N}}_{q} = \overline{{M+N \brack M}}_{q}.
\] 
\item We will omit $q$ from the notation if it is clear from the context that the base is $q$.
\end{enumerate}
\end{remarks}

For example, from Theorem \ref{overGthm} we find that
\[
\overline{{6 \brack 3}} = 1 + 2q + 4q^2 + 8q^3 + 10q^4 + 12q^5 + 12q^6 + 8q^7 + 4q^8 +2q^9,
\]
and we can check that there are 12 overpartitions of $5$ fitting inside a $3 \times 3$ rectangle as follows.
\[
\begin{gathered}
3+2, \overline{3} +2, 3 + \overline{2}, \overline{3}+\overline{2}, 3+1+1, \overline{3}+1+1, 3+1+\overline{1} \\
\overline{3}+1+\overline{1}, 2+2+1, 2+\overline{2}+1, 2+2+\overline{1}, 2+\overline{2}+\overline{1}.
\end{gathered}
\]

Just as $q$-binomial coefficients satisfy simple recurrences, which are the $q$-analogues of Pascal's identity 
\begin{align*}
{M+N \brack N} &= {M+N-1 \brack N}+ q^M {M+N-1 \brack N-1}, \\
{M+N \brack N} &= {M+N -1  \brack N -1}+ q^N {M+N-1 \brack N},
\end{align*}
over $q$-binomial coefficients also satisfy similar recurrences.

\begin{theorem} \label{PropPascal}
For positive integers $M$ and $N$, we have
\begin{enumerate}
\item \begin{equation}
\label{pa1}
\overline{{M+N \brack N}} = \overline{{M+N -1\brack N-1}} +q^N \overline{{M+N-1 \brack N}} + q^N \overline{{M+N-2 \brack N-1}}.
\end{equation}
\item 
\begin{equation}
\label{pa2}
\overline{{M+N \brack N}} = \overline{{M+N-1 \brack N}} + q^M \overline{{M+N-1 \brack N-1}} + q^M \overline{{M+N-2 \brack N-1}}.
\end{equation}
\end{enumerate}
\end{theorem}

By employing over $q$-binomial coefficients, we can establish various identities. We discuss these applications in Section \ref{apps}. Here, we highlight that over $q$-binomial coefficients can be used to derive a Rogers-Ramanujan type theorem for overpartitions. The first Rogers-Ramanujan identity is given by
\[
\sum_{n=0}^{\infty} \frac{q^{n^2}}{(q)_{n}} = \frac{1}{(q;q^5)_{\infty}(q^4;q^5)_{\infty}}.
\]
The left-hand side can be interpreted as the generating function for partitions with a gap $\geq 2$ between two successive parts, and the right-hand side as the generating function for partitions into parts $\equiv 1, 4 \pmod{5}$.  Rogers-Ramanujan identities have been proved via various methods. Among them, one of the most elementary and beautiful is a proof by Andrews which uses recurrence relations of $q$-binomial coefficients \cite{A1, A2}.  Motivated by this proof, we find a Rogers-Ramanujan type identity for overpartitions. Before stating the result, we define three partition functions. Let $A(n)$ be the number of overpartitions $\lambda_1 + \dots + \lambda_{\ell}$ of $n$ satisfying the following gap conditions.
\[
\la_i - \la_{i+1} \geq
\begin{cases}
1, &\text{if $\la_i$ is not overlined,}\\
2, &\text{if $\la_i$ is overlined.}
\end{cases}
\]
(If there are $\ell$ parts in the overpartition, we define $\la_{\ell+1} = 0$ for convenience, thus $\overline{1}$ cannot be a part.) We define $B(n)$ as the number of overpartitions of $n$ with non-overlined parts $\equiv 2 \pmod{4}$ and $C(n)$ as the number of partitions into parts $\not\equiv 0 \pmod{4}$, i.e. the number of $4$-regular partitions of $n$.

\begin{theorem} \label{RRthm}
For all non-negative integers $n$, 
\[
A(n) = B(n) = C(n).
\]
\end{theorem}

\begin{remark}
This is a special case of \cite[Theorem 1.2]{L0}, which is generalized by Chen, Sang, and Shi \cite{CSS}. While the previous results are obtained by employing Bailey chain machinery, we use the recurrence formulas for over $q$-binomial coefficients.
\end{remark}

 The equality $B(n) = C(n)$ is clear from Euler's partition theorem (the number of partitions into odd parts equals the number of partitions into distinct parts), thus the important equality is $A(n)= B(n)$. Here we illustrate Theorem~\ref{RRthm} for the case $n=8$. There are 16 overpartitions satisfying the gap conditions:
\begin{equation*}
\begin{gathered}
8, \overline{8},  7+1, \overline{7}+1, 6+2, \overline{6}+2, 6+\overline{2}, \overline{6}+\overline{2}, 5+3,  \overline{5}+3, \\
5+\overline{3}, \overline{5}+\overline{3}, 5+2+1, \overline{5}+2+1, 4+3+1, 4+\overline{3}+1,
\end{gathered}
\end{equation*}
and there are also 16 overpartitions satisfying the congruence conditions:
\begin{equation*}
\begin{gathered}
\overline{8}, \overline{7}+\overline{1},\overline{6}+\overline{2},6+2, \overline{6}+2, 6+\overline{2}, \overline{5} + \overline{3} ,\overline{5}+\overline{2}+\overline{1},\overline{5}+2+\overline{1}, \overline{4} + \overline{3}+\overline{1}, \overline{4} + 2+2,\\ \overline{4}+2+\overline{2}, \overline{3} + 2 + 2 + \overline{1},\overline{3} + 2 + \overline{2} + \overline{1},  2+2+2+2, 2+2+2+\overline{2}.
\end{gathered}
\end{equation*}

 The rest of paper is organized as follows. In Section 2, we prove Theorem \ref{overGthm} and the recurrence formulas of Theorem~\ref{PropPascal}. In Section \ref{apps}, we give several applications of over $q$-binomial coefficients. In Section 4, by using recurrence formulas we prove a Rogers-Ramaujan type identity for overpartitions.

\section{Basic Properties of over $q$-binomial coefficients}

We start with the proof of Theorem \ref{overGthm}. 

\begin{proof}[Proof of Theorem \ref{overGthm}]
Let $\overline{G}(M,N,k)$ be the generating function for overpartitions fitting inside an $M \times N$ rectangle and having exactly $k$ overlined parts. Such an overpartition can be decomposed as a partition into $k$ distinct parts, each of which is at most $M$, and a partition fitting inside an $M \times (N-k)$ box. By appending a partition fitting into an $(M-k) \times k$ box (generated by ${M \brack k}_q$) to the right of the staircase partition $(k,k-1,\ldots,1)$ (generated by $q^{\frac{k(k+1)}{2}}$) , we see that 
\[
q^{\frac{k(k+1)}{2}} {M \brack k}_q
\]
generates partitions into $k$ distinct parts $\le M$. As ${N+M-k \brack N-k}_q$ generates the partitions fitting inside $M \times (N-k)$ box, we see that
\[
\overline{G}(M,N,k) =  q^{\frac{k(k+1)}{2}} {M \brack k}_q {N+M-k \brack N-k}_q
= q^{\frac{k(k+1)}{2}} \frac{(q)_{M+N-k}}{(q)_k(q)_{M-k}(q)_{N-k}}.
\]
Since $\overline{G}(M,N,k)$ is non-zero if and only if $0 \leq k \leq \min\{M,N\}$, we have
\[
\overline{{N+M \brack N}}_q =\sum_{k=0}^{\min\{M,N\}} \overline{G}(M,N,k) = \sum_{k=0}^{\min\{M,N\}} q^{\frac{k(k+1)}{2}} \frac{(q)_{M+N-k}}{(q)_k(q)_{M-k}(q)_{N-k}}.
\]
\end{proof}

Now we turn to proving the recurrences. 
\begin{proof}[Combinatorial proof of Theorem~\ref{PropPascal}]
 Let $O(M,N,n)$ denote the number of overpartitions of $n$ fitting inside an $M \times N$ rectangle.  Note that $O ( M,N,n) - O( M,N-1,  n)$ is the number of overpartitions of $n$ fitting inside an $M \times N$ rectangle having exactly $N$ parts. 
Let $\la$ be such an overpartition. If the smallest part of $\la$ is $\overline{1}$, then by removing $1$ from every part we obtain an overpartition of $n-N$ fitting inside $(M-1) \times (N-1)$ rectangle. If the smallest part of $\la$ is different from $\overline{1}$, by removing $1$ from every part we arrive at an overpartition of $n-N$ fitting inside an $(M-1) \times N$ rectangle. Therefore, we find that 
\[
O ( M,N,n) - O( M,N-1,  n) =O(M - 1, N - 1, n - N) + O(M - 1,N, n - N) .
\]
By rewriting the above identity in terms of generating functions we obtain the first recurrence.

The second recurrence follows from a similar argument by tracking the size of the maximum part instead of the number of parts.
Note that $O ( M,N,n) - O( M-1,N,  n)$ is the number of overpartitions of $n$ fitting inside an $M \times N$ rectangle with largest part equal to $M$.
If the largest part is overlined, then by removing it we obtain an overpartition of $n-M$ fitting inside a $(M-1) \times (N-1)$ rectangle. If the largest part is not overlined, then by removing it we obtain an overpartition of$n-M$ fitting inside a $M \times (N-1)$ rectangle.
Therefore, we find that 
\[
O ( M,N,n) - O( M-1,N,  n) = O(M-1 ,N-1, n - M)+O(M, N - 1, n - M).
\]
By rewriting the above identity in terms of generating functions we obtain the second recurrence.
\end{proof}

\begin{proof}[Analytic Proof of Theorem \ref{PropPascal}]
We first note that $q$-trinomial coefficients satisfy the following recurrence.
\[
{a+b+c \brack a,b,c}= {a+b+c-1 \brack a-1,b,c} + q^a {a+b+c-1 \brack a,b-1,c} + q^{a+b} {a+b+c-1 \brack a,b,c-1}.
\]
Therefore, we find that
\begin{align*}
\overline{{M+N \brack N }} &= \sum_{k=0}^{\min \{M, N \} } q^{k(k+1)/2} \( {M+N-k-1 \brack N-k-1,k, M-k } + q^{N-k} { M+N-k-1 \brack N-k,k-1,M-k } \right. \\
&\left.  \qquad\qquad + q^{N} { M+N-k-1 \brack N-k,k,M-k-1} \) \\
&=\sum_{k=0}^{\min\{M,N-1\}} q^{k(k+1)/2} {M+N-k-1 \brack k, M-k, N-1 } \\
&\quad+ \sum_{k=0}^{\min\{M-1,N-1\}} q^{k(k+1)/2 + N} { M+N-k-2 \brack k, M-1-k,N-1-k } \\
&\quad + \sum_{k=0}^{\min\{M-1,N\}} q^{k(k+1)/2 + N} { M+N-k-1 \brack k,M-1-k,N-1} \\
&=\overline{{M+N -1 \brack N-1}}+ q^N \overline{{M+N -2 \brack N-1}}+ q^N \overline{{M+N -1 \brack N}},  
\end{align*}
where we have made a change of variable $k \rightarrow k+1$ in the second sum. The second recurrence can be proved similarly.
\end{proof}

Throughout the paper, we use the following asymptotic behaviour frequently.

\begin{proposition} \label{Njlim}
For a non-negative integer $j$,
\[
\lim_{N \rightarrow \infty} \overline{{N \brack j}}  = \frac{(-q)_{j}}{(q)_{j}}.
\]
\end{proposition}

\begin{proof}
When $N$ goes to the infinity, the restriction on the number of parts disappears. 
\end{proof}

Proposition~\ref{Njlim} is useful to obtain new identities.  For example, by taking a limit $N \to \infty$ in Theorem \ref{overGthm}, we find that
\[
\sum_{k=0}^{j} \frac{q^{k(k+1)/2}}{(q)_{k}(q)_{j-k}} = \frac{(-q)_{j}}{(q)_{j}},
\]
which gives an alternative generating function for overpartitions into parts $\leq j$. This identity is also a special case of  the finite $q$-binomial theorem \cite[Exer. 1.2(vi)]{GR}.

\section{Applications} \label{apps}

By tracking the number of parts in the overpartitions, we prove the following identity.
\begin{proposition}
For a positive integer $N$, 
\[
\frac{(-zq)_{N}}{(zq)_{N}} = 1+ \sum_{k \geq 1} z^k q^k \left( \overline{\left[ \begin{matrix} N+k-1 \\ k \end{matrix} \right] } + \overline{\left[ \begin{matrix} N+k-2 \\ k-1 \end{matrix} \right] }  \right).
\]
\end{proposition}

\begin{proof}
Let $\overline{p}_{N} (n,k)$ be the number of overpartitions of $n$ into parts $\leq N$ with $k$ parts. Then, it is not hard to see that
\[
\frac{(-zq)_{N}}{(zq)_{N}}  = \sum_{n \geq 0} \sum_{k \geq 0} \overline{p}_{N} (n,k) z^k q^n .
\]

Let $\lambda$ be an overpartition counted by $\overline{p}_{N} (n,k)$. Discussing whether the smallest part of $\lambda$ is equal to $\overline{1}$ and removing $1$ from each part as in the proof of Theorem~\ref{PropPascal}, we have
$$\overline{p}_{N} (n,k) = \overline{p}_{N-1} (n-k,k) + \overline{p}_{N-1} (n-k,k-1).$$
Thus $$\sum_{n \geq 0}\overline{p}_{N} (n,k) q^n = q^k \left( \overline{\left[ \begin{matrix} N+k-1 \\ k \end{matrix} \right] } + \overline{\left[ \begin{matrix} N+k-2 \\ k-1 \end{matrix} \right] }  \right).$$
The claimed identity follows.
 \end{proof}

By taking the limit as $N \rightarrow \infty$ in the above proposition, we find the following generating function.
\begin{corollary}
Let $\overline{p} (n,k)$ be the number of overpartitions of $n$ with $k$ parts. Then,
\[
\sum_{n \geq 0} \sum_{k \geq 0} \overline{p} (n,k) z^k q^n =\frac{(-zq)_{\infty}}{(zq)_{\infty}} = 1+2 \sum_{k \geq 1} \frac{ z^k q^{k} (-q)_{k-1}}{(q)_{k}}.
\]
\end{corollary}

\begin{remark}
The above identity is a special case of the $q$-binomial theorem \cite[(II.3)]{GR}. 
\end{remark}

Note that
\[
\sum_{k \geq 1} \frac{ q^{k} (-q)_{k-1}}{(q)_{k}} \equiv \sum_{k \geq 1} \frac{ q^{k} }{1-q^k} = \sum_{n \geq 1} \tau (n) q^n  \pmod{2},
\]
where $\tau(n)$ is the number of divisors of $n$. This recovers a well known congruence. 
\begin{corollary}
For all non-negative integers $n$, 
\[
\overline{p}(n) \equiv 2 \tau(n) \pmod{4}.
\]
\end{corollary}

 Our next application is finding an analogue of Sylvester's identity \cite{Syl}:
 \begin{align*}
 (-xq)_{N} = 1 + &\sum_{j \geq 1} \left[ \begin{matrix} N+1-j \\ j \end{matrix} \right] (-xq;q)_{j-1} x^{j} q^{3j(j-1)/2} \\
 &+\sum_{j \geq 1} \left[ \begin{matrix} N-j \\ j \end{matrix} \right] (-xq;q)_{j-1} x^{j+1} q^{3j(j+1)/2}.
 \end{align*}
 
 We define $\overline{S}(N;x;q)$ as
\[
\overline{S}(N;x,q) := 1 + \sum_{j \geq 1} \left(\overline{\left[ \begin{matrix} N-1 \\ j-1 \end{matrix} \right] } \frac{(-xq)_{j-1}}{(xq)_{j-1}} x^j q^{j^2} + \overline{\left[ \begin{matrix} N \\ j \end{matrix} \right] } \frac{(-xq)_{j}}{(xq)_{j}} x^j q^{j^2} \right).
\]
Then we have the following identity.

\begin{theorem} 
For a positive integer $N$, 
\[
\overline{S}(N;x;q)= \frac{(-xq)_{N}}{(xq)_{N}}.
\]
\end{theorem}

\begin{proof}
Let us consider an overpartition into parts $\leq N$, generated by $\frac{(-xq)_{N}}{(xq)_{N}}$. The variable $j$ counts the size of the Durfee square of the overpartition.  The Durfee square is generated $x^jq^{j^2}$. Then either the corner at the bottom right of the Durfee square is overlined or it is not. If it is overlined, then we have an overpartition generated by $\overline{\left[ \begin{matrix} N-1 \\ j-1 \end{matrix} \right]}$ at the right of the Durfee square, and an overpartition generated by $\frac{(-xq)_{j-1}}{(xq)_{j-1}}$ under it. If it is not overlined, then we have an overpartition generated by $\overline{\left[ \begin{matrix} N \\ j \end{matrix} \right]}$ to the right of the Durfee square, and an overpartition generated by $\frac{(-xq)_{j}}{(xq)_{j}}$ under it.
These two cases correspond to the two sums in $\overline{S}(N;x;q)$.
\end{proof}

By taking a limit $j \rightarrow \infty$ and using Proposition \ref{Njlim}, we obtain the following identity.
\begin{corollary} \label{corSyl}
We have 
\[
\frac{(-xq)_{\infty}}{(xq)_{\infty}} = 1 + \sum_{j \geq 1} \left( \frac{(-q)_{j-1}}{(q)_{j-1}} \frac{(-xq)_{j-1}}{(xq)_{j-1}} x^j q^{j^2} + \frac{(-q)_{j}}{(q)_{j}} \frac{(-xq)_{j}}{(xq)_{j}} x^j q^{j^2} \right).
\]
\end{corollary}

In particular, by setting $x=-1$ we obtain a well known theta function identity.  
\[
\frac{(q)_{\infty}}{(-q)_{\infty}} = \sum_{n \in \mathbb{Z} } (-1)^n q^{n^2}.
\]

 As another application, we obtain the overpartition rank generating function. To explain Ramanujan's famous three partition congruences, Dyson \cite{dyson} introduced the rank for the partition as the difference between the size of the largest part and the number of parts. For an overpartition, we can define a rank in the same way \cite{BL3}. Let $\overline{N} (m,n)$ be the number of overpartitions of $n$ with rank $m$. Then, we can express the generating function in terms of over $q$-binomial coefficients.
 
\begin{theorem}
For a non-negative integer $m$, 
\begin{align*}
N_m (q) &:= \sum_{n \geq 0} \overline{N}(m,n) q^n \\
&= 2q^{1+m} + \sum_{k\geq 2} q^{2k+m-1} \left( \overline{\left[ \begin{matrix} 2k+m-2 \\ k-1 \end{matrix} \right]}+  \overline{\left[ \begin{matrix} 2k+m-3 \\ k-1 \end{matrix} \right]} + \right. \\
 &\quad\quad  \left.  \overline{\left[ \begin{matrix} 2k+m-3 \\ k-2 \end{matrix} \right]} + \overline{\left[ \begin{matrix} 2k+m-4 \\ k-2 \end{matrix} \right]} \right).
\end{align*}
\end{theorem} 

\begin{proof}
If there is only one part in the overpartition, $m+1$ and $\overline{m+1}$ are the only two such overpartitions with rank $m$, which corresponds to $2q^{m+1}$. Now we assume that an overpartition has at least two parts and the rank of the overparition is $m$. Under this assumption, the largest part would be $m+k$ and the number of part is $k$, this corresponds to $q^{m+2k-1}$ inside the summation. Now the first sum counts the case where the largest part is not overlined and there is no $\overline{1}$. The second sum counts the case where the largest part is overlined and there is no $\overline{1}$.
The third sum counts the case where the largest part is not overlined and the smallest part is $\overline{1}$. The last sum corresponds to the case where the largest part is overlined and the smallest part is $\overline{1}$.
\end{proof}
 
 By comparing the known generating function for $\overline{N}_m(q)$  \cite[Proposition 3.2]{L1} 
 \[
 N_m (q) = 2 \frac{(-q)_{\infty}}{(q)_{\infty}} \sum_{n \geq 1} \frac{(-1)^{n-1} q^{n^2 + |m| n} (1-q^n)}{1+q^n},
\]
 we derive the following identity.
 
 \begin{corollary}
 For a non-negative integer $m$, 
 \begin{align*}
 &2 \frac{(-q)_{\infty}}{(q)_{\infty}} \sum_{n \geq 1} \frac{(-1)^{n-1} q^{n^2 + |m| n} (1-q^n)}{1+q^n} \\
 &=2q^{1+m} + \sum_{k\geq 2} q^{2k+m-1} \left( \overline{\left[ \begin{matrix} 2k+m-2 \\ k-1 \end{matrix} \right]}+  \overline{\left[ \begin{matrix} 2k+m-3 \\ k-1 \end{matrix} \right]}  \right. \\
 &\quad\quad\quad\quad\quad\quad  \left.  +\overline{\left[ \begin{matrix} 2k+m-3 \\ k-2 \end{matrix} \right]} + \overline{\left[ \begin{matrix} 2k+m-4 \\ k-2 \end{matrix} \right]} \right).
 \end{align*}
 \end{corollary}

\section{Proof of a Rogers-Ramanujan type identity}

We first define two functions
\begin{align*}
\overline{D}(N,x;q) &:= \sum_{j \geq 0} \overline{\left[ \begin{matrix} N \\ j \end{matrix} \right]} x^j q^{j(j+1)/2}
\intertext{and}
\overline{C}(N,x;q)&:=\sum_{j \geq 0} \overline{\left[ \begin{matrix} N \\ j \end{matrix} \right]}  \frac{(xq)_{j}}{(-xq)_{j}} (-1)^j x^{2j} \left( q^{j(2j+1)}  - x q^{(j+1)(2j+1)} \right).
\end{align*}

The following observation is the key for obtaining a Rogers-Ramanujan type identity.

\begin{theorem} \label{keyobprop}
For a positive integer $N$, 
\[
\frac{(xq)_{N}}{(-xq)_{N}} \overline{D}(N,x;q) - \overline{C}(N,x;q) \in x^2 q^{N+3} \cdot \mathbb{Z}[[x,q]].
\]
\end{theorem}

By taking the limit as $N \to \infty$, we obtain the following corollary.
\begin{corollary} \label{RRcor}
We have 
\[
\frac{(xq)_{\infty}}{(-xq)_{\infty}} \overline{D}(\infty,x;q)  = \overline{C}(\infty,x;q).
\]
\end{corollary}

In particular, the case $x=1$ is a Rogers-Ramanujan type identity, where we applied Lemma \ref{Njlim} to evaluate the limit. 

\begin{corollary}
We have
\[
\frac{(q)_{\infty}}{(-q)_{\infty}} \overline{D}(\infty,1 ; q) = \sum_{n \in \mathbb{Z} } (-1)^n q^{n(2n+1)}=(q,q^3,q^4;q^4)_{\infty} .
\]
\end{corollary}

\begin{proof}[Proof of Theorem \ref{RRthm}]
After multiplying $\frac{(-q)_{\infty}}{(q)_{\infty}}$ to both sides and from the definitions, we obtain that
\[
\sum_{k=0}^{\infty} \frac{q^{k(k+1)/2} (-q)_{k}}{(q)_{k}} = \frac{(-q)_\infty}{(q^2;q^4)_{\infty}} = \frac{1}{(q,q^2,q^3;q^4)_{\infty}}.
\]
A basic partition theoretic interpretation of the above identity gives the desired result.
\end{proof}

Now we turn to proving Theorem \ref{keyobprop}. Let $\overline{g} (N,x) :=  \frac{(xq)_{N}}{(-xq)_{N}} \overline{D}(N,x)$. The key idea of the proof is that $\overline{g} (N,x)$ and $\overline{C} (N,x)$ satisfy the same recurrence (up to a high power of $q$ times a polynomial in $x$ and $q$) as follows.

\begin{lemma} \label{recoverg}
\[
\overline{g} (N,x) = (1-xq) \overline{g} (N-1,xq) + \frac{(xq)_{2}}{(-xq)_{2}} xq^2 \overline{g} (N-2,xq^2).
\]
\end{lemma}

\begin{lemma}
\label{recC}
\begin{equation}
\label{receq}
\overline{C} (N,x) - (1-xq) \overline{C} (N-1,xq) - \frac{(xq)_{2}}{(-xq)_{2}} xq^2 \overline{C} (N-2,xq^2) \in x^2 q^{N+3} \cdot \mathbb{Z}[[x,q]].
\end{equation}
\end{lemma}

Theorem \ref{keyobprop} follows immediately from these two recurrences and an induction over $N$. We now need to prove these lemmas.

\begin{proof}[Proof of Lemma \ref{recoverg}]
By applying the first recurrence in Theorem \ref{PropPascal}, we find that
\begin{align*}
\overline{D}(N,x ) &= \sum_{j \geq 0} 
\( \overline{\left[ \begin{matrix} N-1 \\ j-1 \end{matrix} \right] } +q^j  \overline{\left[ \begin{matrix} N-1 \\ j \end{matrix} \right] } +q^j \overline{\left[ \begin{matrix} N-2 \\ j-1 \end{matrix} \right] } \)  x^j q^{j(j+1)/2} \\
&= \sum_{j \geq 0}  \overline{\left[ \begin{matrix} N-1 \\ j \end{matrix} \right] } x^{j+1} q^{(j+1)(j+2)/2} + \overline{D}(N-1,xq) + \sum_{j \geq 0}  \overline{\left[ \begin{matrix} N-2 \\ j \end{matrix} \right] } x^{j+1} q^{(j+1)(j+4)/2} \\
&= (1+xq) \overline{D} (N-1,xq) + xq^2 \overline{D} (N-2,xq^2 ),
\end{align*}
where we replace $j-1$ by $j$ in the first and the third sum for the second identity. After multiplying by $\frac{(xq)_{N}}{(-xq)_{N}}$ we get the desired recurrence.
\end{proof}

\begin{proof}[Proof of Lemma \ref{recC}]
We calculate each term in~\eqref{receq}. By the definition of $\overline{C}$, we find that
\[
(1-xq) \overline{C} (N-1,xq) = (1+xq) \sum_{j\geq 0} \overline{\left[ \begin{matrix} N-1 \\ j \end{matrix} \right] } \frac{(xq)_{j+1}}{(-xq)_{j+1}} (-1)^j x^{2j} \left( q^{2j^2+3j}-xq^{2j^2+5j+2}\right).
\]
Expanding and making the change of variable $j \rightarrow j-1$ in the fourth sum, we get
\begin{equation}
\label{C(N-1)}
\begin{aligned}
(1-xq) \overline{C} (N-1,xq) &= \sum_{j\geq 0} \overline{\left[ \begin{matrix} N-1 \\ j \end{matrix} \right] } \frac{(xq)_{j+1}}{(-xq)_{j+1}} (-1)^j x^{2j} q^{2j^2+3j}\\
&+\sum_{j\geq 0} \overline{\left[ \begin{matrix} N-1 \\ j \end{matrix} \right] } \frac{(xq)_{j+1}}{(-xq)_{j+1}} (-1)^{j+1} x^{2j+1} q^{2j^2+5j+2}\\
&+\sum_{j\geq 0}\overline{\left[ \begin{matrix} N-1 \\ j \end{matrix} \right] } \frac{(xq)_{j+1}}{(-xq)_{j+1}} (-1)^j x^{2j+1} q^{2j^2+3j+1}\\
&+\sum_{j\geq 1} \overline{\left[ \begin{matrix} N-1 \\ j-1 \end{matrix} \right] } \frac{(xq)_{j}}{(-xq)_{j}} (-1)^{j} x^{2j} q^{2j^2+j}.
\end{aligned}
\end{equation}
By the change of variable $j \rightarrow j-1$ in the definition of $\frac{(xq)_{2}}{(-xq)_{2}} xq^2 \overline{C} (N-2,xq^2)$, we obtain
\begin{equation}
\label{C(N-2)}
\begin{aligned}
\frac{(xq)_{2}}{(-xq)_{2}} xq^2 \overline{C} (N-2,xq^2) &= \sum_{j\geq 1} \overline{\left[ \begin{matrix} N-2 \\ j-1 \end{matrix} \right] } \frac{(xq)_{j+1}}{(-xq)_{j+1}} (-1)^{j+1} x^{2j-1} q^{2j^2+j-1} \\
&+\sum_{j\geq 1} \overline{\left[ \begin{matrix} N-2 \\ j-1 \end{matrix} \right] } \frac{(xq)_{j+1}}{(-xq)_{j+1}} (-1)^{j} x^{2j} q^{2j^2+3j}.
\end{aligned}
\end{equation}
Using the first recurrence~\eqref{pa1} on the first sum in~\eqref{C(N-1)} and the second sum in~\eqref{C(N-2)} and extracting the term $j=0$ in the first and third sums of~\eqref{C(N-1)} leads to
\begin{align}
(1-xq) &\overline{C} (N-1,xq) +\frac{(xq)_{2}}{(-xq)_{2}} xq^2 \overline{C} (N-2,xq^2) \nonumber\\
=  1-xq &+\sum_{j\geq 1} \overline{\left[ \begin{matrix} N \\ j \end{matrix} \right] } \frac{(xq)_{j+1}}{(-xq)_{j+1}} (-1)^j x^{2j} q^{2j^2+2j} \nonumber \\
&+\sum_{j\geq 1} \overline{\left[ \begin{matrix} N-1 \\ j-1 \end{matrix} \right] } \frac{(xq)_{j+1}}{(-xq)_{j+1}} (-1)^{j+1} x^{2j} q^{2j^2+2j} \nonumber\\
&+\sum_{j\geq 0} \overline{\left[ \begin{matrix} N-1 \\ j \end{matrix} \right] } \frac{(xq)_{j+1}}{(-xq)_{j+1}} (-1)^{j+1} x^{2j+1} q^{2j^2+5j+2} \label{C(N-1)+C(N-2)}
 \\
&+\sum_{j\geq 1}\overline{\left[ \begin{matrix} N-1 \\ j \end{matrix} \right] } \frac{(xq)_{j+1}}{(-xq)_{j+1}} (-1)^j x^{2j+1} q^{2j^2+3j+1}\nonumber\\
&+\sum_{j\geq 1} \overline{\left[ \begin{matrix} N-1 \\ j-1 \end{matrix} \right] } \frac{(xq)_{j}}{(-xq)_{j}} (-1)^{j} x^{2j} q^{2j^2+j}\nonumber\\
&+\sum_{j\geq 1} \overline{\left[ \begin{matrix} N-2 \\ j-1 \end{matrix} \right] } \frac{(xq)_{j+1}}{(-xq)_{j+1}} (-1)^{j+1} x^{2j-1} q^{2j^2+j-1}.\nonumber
\end{align}

Now we want to write both $\overline{C} (N,x)$ and $(1-xq) \overline{C} (N-1,xq) +\frac{(xq)_{2}}{(-xq)_{2}} xq^2 \overline{C} (N-2,xq^2)$ as sums involving the product $\frac{(xq)_{j}}{(-xq)_{j+1}}$ to be able to make cancellations. We have
\begin{equation*}
\begin{aligned}
\overline{C} (N,x) &= \sum_{j\geq 0} \overline{\left[ \begin{matrix} N \\ j \end{matrix} \right] } \frac{(xq)_{j}}{(-xq)_{j+1}} \left(1+xq^{j+1}\right) (-1)^{j} x^{2j} q^{2j^2+j}\\
&+\sum_{j\geq 0} \overline{\left[ \begin{matrix} N \\ j \end{matrix} \right] } \frac{(xq)_{j}}{(-xq)_{j+1}} \left(1+xq^{j+1}\right) (-1)^{j+1} x^{2j+1} q^{2j^2+3j+1}.
\end{aligned}
\end{equation*}
Extracting the terms $j=0$ of each sum and expanding, we get
\begin{align}
\overline{C} (N,x) &= 1-xq + \sum_{j\geq 1} \overline{\left[ \begin{matrix} N \\ j \end{matrix} \right] } \frac{(xq)_{j}}{(-xq)_{j+1}} (-1)^{j} x^{2j} q^{2j^2+j} \nonumber \\
&+ \sum_{j\geq 1} \overline{\left[ \begin{matrix} N \\ j \end{matrix} \right] } \frac{(xq)_{j}}{(-xq)_{j+1}}(-1)^{j} x^{2j+1} q^{2j^2+2j+1} \label{C(N)}
 \\
&+\sum_{j\geq 1} \overline{\left[ \begin{matrix} N \\ j \end{matrix} \right] } \frac{(xq)_{j}}{(-xq)_{j+1}} (-1)^{j+1} x^{2j+1} q^{2j^2+3j+1} \nonumber \\
&+\sum_{j\geq 1} \overline{\left[ \begin{matrix} N \\ j \end{matrix} \right] } \frac{(xq)_{j}}{(-xq)_{j+1}} (-1)^{j+1} x^{2j+2} q^{2j^2+4j+2}. \nonumber
\end{align}

Rewriting all the sums in~\eqref{C(N-1)+C(N-2)} except the third one in terms of  $\frac{(xq)_{j}}{(-xq)_{j+1}}$ leads to
\begin{align} 
(1-xq) &\overline{C} (N-1,xq) +\frac{(xq)_{2}}{(-xq)_{2}} xq^2 \overline{C} (N-2,xq^2) \\
=  1-xq &+\sum_{j\geq 1} \overline{\left[ \begin{matrix} N \\ j \end{matrix} \right] } \frac{(xq)_{j}}{(-xq)_{j+1}} (-1)^j x^{2j} q^{2j^2+2j} \nonumber \\
&+\sum_{j\geq 1} \overline{\left[ \begin{matrix} N \\ j \end{matrix} \right] } \frac{(xq)_{j}}{(-xq)_{j+1}} (-1)^{j+1} x^{2j+1} q^{2j^2+3j+1} \nonumber \\
&+\sum_{j\geq 1} \overline{\left[ \begin{matrix} N-1 \\ j-1 \end{matrix} \right] } \frac{(xq)_{j}}{(-xq)_{j+1}} (-1)^{j+1} x^{2j} q^{2j^2+2j} \nonumber \\
&+\sum_{j\geq 1} \overline{\left[ \begin{matrix} N-1 \\ j-1 \end{matrix} \right] } \frac{(xq)_{j}}{(-xq)_{j+1}} (-1)^{j} x^{2j+1} q^{2j^2+3j+1} \nonumber \\
&+\sum_{j\geq 0} \overline{\left[ \begin{matrix} N-1 \\ j \end{matrix} \right] } \frac{(xq)_{j+1}}{(-xq)_{j+1}} (-1)^{j+1} x^{2j+1} q^{2j^2+5j+2}  \label{C(N-1)+C(N-2)bis} \\
&+\sum_{j\geq 1}\overline{\left[ \begin{matrix} N-1 \\ j \end{matrix} \right] } \frac{(xq)_{j}}{(-xq)_{j+1}} (-1)^j x^{2j+1} q^{2j^2+3j+1} \nonumber \\
&+\sum_{j\geq 1}\overline{\left[ \begin{matrix} N-1 \\ j \end{matrix} \right] } \frac{(xq)_{j}}{(-xq)_{j+1}} (-1)^{j+1} x^{2j+2} q^{2j^2+4j+2} \nonumber \\
&+\sum_{j\geq 1} \overline{\left[ \begin{matrix} N-1 \\ j-1 \end{matrix} \right] } \frac{(xq)_{j}}{(-xq)_{j+1}} (-1)^{j} x^{2j} q^{2j^2+j} \nonumber \\
&+\sum_{j\geq 1} \overline{\left[ \begin{matrix} N-1 \\ j-1 \end{matrix} \right] } \frac{(xq)_{j}}{(-xq)_{j+1}} (-1)^{j} x^{2j+1} q^{2j^2+2j+1} \nonumber \\
&+\sum_{j\geq 1} \overline{\left[ \begin{matrix} N-2 \\ j-1 \end{matrix} \right] } \frac{(xq)_{j}}{(-xq)_{j+1}} (-1)^{j+1} x^{2j-1} q^{2j^2+j-1} \nonumber \\
&+\sum_{j\geq 1} \overline{\left[ \begin{matrix} N-2 \\ j-1 \end{matrix} \right] } \frac{(xq)_{j}}{(-xq)_{j+1}} (-1)^{j} x^{2j} q^{2j^2+2j}. \nonumber 
\end{align}

Subtracting~\eqref{C(N)} from~\eqref{C(N-1)+C(N-2)bis} and noting that the third sum of~\eqref{C(N)} cancels with the second sum of~\eqref{C(N-1)+C(N-2)bis} we obtain
\begin{align} 
\overline{C} &(N,x) - (1-xq) \overline{C} (N-1,xq) -\frac{(xq)_{2}}{(-xq)_{2}} xq^2 \overline{C} (N-2,xq^2)  \nonumber\\
&= \sum_{j\geq 1} \overline{\left[ \begin{matrix} N \\ j \end{matrix} \right] } \frac{(xq)_{j}}{(-xq)_{j+1}} (-1)^{j} x^{2j} q^{2j^2+j} \label{eq1}\\
&+ \sum_{j\geq 1} \overline{\left[ \begin{matrix} N \\ j \end{matrix} \right] } \frac{(xq)_{j}}{(-xq)_{j+1}}(-1)^{j} x^{2j+1} q^{2j^2+2j+1} \label{eq2}\\
&+\sum_{j\geq 1} \overline{\left[ \begin{matrix} N \\ j \end{matrix} \right] } \frac{(xq)_{j}}{(-xq)_{j+1}} (-1)^{j+1} x^{2j+2} q^{2j^2+4j+2}  \label{eq3}\\
&+\sum_{j\geq 1} \overline{\left[ \begin{matrix} N \\ j \end{matrix} \right] } \frac{(xq)_{j}}{(-xq)_{j+1}} (-1)^{j+1} x^{2j} q^{2j^2+2j} \label{eq4}\\
&+\sum_{j\geq 1} \overline{\left[ \begin{matrix} N-1 \\ j-1 \end{matrix} \right] } \frac{(xq)_{j}}{(-xq)_{j+1}} (-1)^{j} x^{2j} q^{2j^2+2j} \label{eq5}\\
&+\sum_{j\geq 1} \overline{\left[ \begin{matrix} N-1 \\ j-1 \end{matrix} \right] } \frac{(xq)_{j}}{(-xq)_{j+1}} (-1)^{j+1} x^{2j+1} q^{2j^2+3j+1} \label{eq6}\\
&+\sum_{j\geq 0} \overline{\left[ \begin{matrix} N-1 \\ j \end{matrix} \right] } \frac{(xq)_{j+1}}{(-xq)_{j+1}} (-1)^{j} x^{2j+1} q^{2j^2+5j+2} \label{eq7}\\
&+\sum_{j\geq 1}\overline{\left[ \begin{matrix} N-1 \\ j \end{matrix} \right] } \frac{(xq)_{j}}{(-xq)_{j+1}} (-1)^{j+1} x^{2j+1} q^{2j^2+3j+1} \label{eq8}\\
&+\sum_{j\geq 1}\overline{\left[ \begin{matrix} N-1 \\ j \end{matrix} \right] } \frac{(xq)_{j}}{(-xq)_{j+1}} (-1)^{j} x^{2j+2} q^{2j^2+4j+2} \label{eq9}\\
&+\sum_{j\geq 1} \overline{\left[ \begin{matrix} N-1 \\ j-1 \end{matrix} \right] } \frac{(xq)_{j}}{(-xq)_{j+1}} (-1)^{j+1} x^{2j} q^{2j^2+j} \label{eq10}\\
&+\sum_{j\geq 1} \overline{\left[ \begin{matrix} N-1 \\ j-1 \end{matrix} \right] } \frac{(xq)_{j}}{(-xq)_{j+1}} (-1)^{j+1} x^{2j+1} q^{2j^2+2j+1} \label{eq11}\\
&+\sum_{j\geq 1} \overline{\left[ \begin{matrix} N-2 \\ j-1 \end{matrix} \right] } \frac{(xq)_{j}}{(-xq)_{j+1}} (-1)^{j} x^{2j-1} q^{2j^2+j-1} \label{eq12}\\
&+\sum_{j\geq 1} \overline{\left[ \begin{matrix} N-2 \\ j-1 \end{matrix} \right] } \frac{(xq)_{j}}{(-xq)_{j+1}} (-1)^{j+1} x^{2j} q^{2j^2+2j}. \label{eq13}
\end{align}

By the second recurrence~\eqref{pa2}, we observe that the sum~\eqref{eq6} is equal to
\begin{equation}
\label{eq6'}
\sum_{j\geq 1} \overline{\left[ \begin{matrix} N-2 \\ j-1 \end{matrix} \right] } \frac{(xq)_{j}}{(-xq)_{j+1}} (-1)^{j+1} x^{2j+1} q^{2j^2+3j+1} +O\left(x^3q^{N+5}\right),
\end{equation}
where we define $f(x,q)= O (x^k q^\ell)$ to mean that $f(x,q)\in x^k q^{\ell} \mathbb{Z} [[x,q]]$.  Thus by the first recurrence~\eqref{pa1}, the sum of~\eqref{eq2},~\eqref{eq6},~\eqref{eq8} and~\eqref{eq11} is equal to $O\left(x^3q^{N+5}\right)$. Furthermore, by the second recurrence~\eqref{pa2}, the sum of~\eqref{eq3} and~\eqref{eq9} is $O\left(x^4q^{N+7}\right).$ Finally again by the second recurrence~\eqref{pa2}, the sum~\eqref{eq4} is equal to
$$\sum_{j\geq 1} \overline{\left[ \begin{matrix} N-1 \\ j \end{matrix} \right] } \frac{(xq)_{j}}{(-xq)_{j+1}} (-1)^{j+1} x^{2j} q^{2j^2+2j} + O\left(x^2 q^{N+3}\right).$$
Thus by the first recurrence~\eqref{pa1}, the sum of~\eqref{eq1},~\eqref{eq4},~\eqref{eq10} and~\eqref{eq13} is equal to $O\left(x^2q^{N+3}\right).$

Hence we are left with the following
\begin{align*} 
\overline{C} &(N,x) - (1-xq) \overline{C} (N-1,xq) -\frac{(xq)_{2}}{(-xq)_{2}} xq^2 \overline{C} (N-2,xq^2) \\
&=\sum_{j\geq 1} \overline{\left[ \begin{matrix} N-1 \\ j-1 \end{matrix} \right] } \frac{(xq)_{j}}{(-xq)_{j+1}} (-1)^{j} x^{2j} q^{2j^2+2j}\\
&+\sum_{j\geq 0} \overline{\left[ \begin{matrix} N-1 \\ j \end{matrix} \right] } \frac{(xq)_{j+1}}{(-xq)_{j+1}} (-1)^{j} x^{2j+1} q^{2j^2+5j+2}\\
&+\sum_{j\geq 1} \overline{\left[ \begin{matrix} N-2 \\ j-1 \end{matrix} \right] } \frac{(xq)_{j}}{(-xq)_{j+1}} (-1)^{j} x^{2j-1} q^{2j^2+j-1}\\
&+O\left(x^2q^{N+3}\right).
\end{align*}

By the second recurrence~\eqref{pa2}, the third sum is equal to
$$\sum_{j\geq 1} \overline{\left[ \begin{matrix} N-1 \\ j-1 \end{matrix} \right] } \frac{(xq)_{j}}{(-xq)_{j+1}} (-1)^{j} x^{2j-1} q^{2j^2+j-1} +O\left(x^3q^{N+7}\right).$$
Factorising it with the first sum we get
\begin{align*}
\overline{C} &(N,x) - (1-xq) \overline{C} (N-1,xq) -\frac{(xq)_{2}}{(-xq)_{2}} xq^2 \overline{C} (N-2,xq^2) \\
&=\sum_{j\geq 0} \overline{\left[ \begin{matrix} N-1 \\ j \end{matrix} \right] } \frac{(xq)_{j+1}}{(-xq)_{j+1}} (-1)^{j} x^{2j+1} q^{2j^2+5j+2}\\
&+\sum_{j\geq 1} \overline{\left[ \begin{matrix} N-1 \\ j-1 \end{matrix} \right] } \frac{(xq)_{j}}{(-xq)_{j}} (-1)^{j} x^{2j-1} q^{2j^2+j-1}\\
&+O\left(x^2q^{N+3}\right).
\end{align*}
Now by a simple change of variable $j \rightarrow j-1$ we see that the two sums are cancelled, and this completes the proof. 
\end{proof}

\section{Concluding Remarks}
 In this paper, we introduced a polynomial version of the overpartition generating function and its applications in the theory. The main theme is emphasizing combinatorial motivations and roles of recurrence formulas to derive results. In this sense, the finite form of the identities are the main objects of this paper. The limiting version of the identities can also be proven using well-known transformation formulas in the theory of basic hypergeometric series. We can recover Corollary~\ref{corSyl} by setting $a=xq$, $b=-q$, and $c,d \to \infty$ in the very-well-poised $_6 \phi_5$ summation \cite[(II.20)]{GR}. Corollary~\ref{RRcor} can also be proven via employing $_8 \phi_7$ summation and Heine's transformation. By setting $a=xq$, $b=-q$, and $c,d,e,f \to \infty$ in $_8 \phi_7$ summation \cite[(III.23)]{GR}, we find that
 \[
 \overline{C}(\infty,x ; q) = (xq)_{\infty} \lim_{e,f \to \infty}  {_2 \phi_1 } (e,f ; -xq ; q , xq^2/ef).
 \]
 By employing Heine's transformation \cite[(III.2)]{GR}, we can derive that the limit in the above equation is the same as
 %(A=e, B=f, C=aq/b, Z=aq/ef)
 \[
 \frac{1}{(-xq)_{\infty}} \overline{D} (\infty,x ; q).
 \]

\section*{Acknowledgement}
The authors thank Jeremy Lovejoy for the valuable discussions and comments at every stage of this paper.

\end{document}